\documentclass[11pt]{amsart}
\usepackage[utf8]{inputenc}
\usepackage{amsfonts,amssymb,amsbsy,amsmath} %
\usepackage{bm} %
\usepackage[inline,shortlabels]{enumitem} %
\usepackage{amsthm} %
\usepackage{cleveref} %
\usepackage{xcolor}
\usepackage{comment}
\usepackage{url}
\usepackage{cleveref}

\theoremstyle{plain}
\newtheorem{theorem}{Theorem}
\newtheorem{lemma}[theorem]{Lemma}

\newtheorem{remark}[theorem]{Remark}
\newtheorem{corollary}[theorem]{Corollary}

\theoremstyle{definition}

\newcommand{\DEF}[1]{\emph{\color{blue} #1}}
\newcommand{\B}[1]{\mathbb{#1}}

\newcommand{\excessOps}{A}
\DeclareMathOperator*{\Ass}{Ass}

\DeclareMathOperator*{\Frac}{Frac}
\DeclareMathOperator*{\nil}{nil}

\newcommand{\p}{\partial}
\newcommand{\ideal}[1]{\langle #1 \rangle}
\DeclareMathOperator*{\Span}{span}

\newcommand{\KK}{{\B K}}

\newcommand{\mm}{\mathfrak{m}}
\newcommand{\pp}{\mathfrak{p}}

\newcommand{\xx}{{\boldsymbol{x}}}
\newcommand{\yy}{{\boldsymbol{y}}}
\renewcommand{\tt}{{\boldsymbol{t}}}

\newcommand{\rat}[2]{#2^{(#1)}}

\title{Local dual spaces and primary decomposition
}
\author{Justin Chen, Marc H\"ark\"onen, Anton Leykin}
\thanks{Partially supported by NSF DMS award~\#2001267.}

\begin{document}
\address{School of Mathematics, Georgia Tech, Atlanta GA, USA}

\begin{abstract}
  Generalizing the concept of the Macaulay inverse system, we introduce a way to describe localizations of an ideal in a polynomial ring. This leads to an approach to the differential primary decomposition as a description of the affine scheme defined by the ideal.  
\end{abstract}
\maketitle

This paper is dedicated to the memory of Agnes Szanto.
She had a large impact on the community of symbolic computation as witnessed in part by her work cited in relation to this article.
Agnes is greatly missed.

\section{Introduction}

Let $R := \KK[\xx] = \KK[x_1,\dotsc,x_n] $ be a polynomial ring over a field $\KK$ of characteristic $0$, and $I \subseteq R$ an $R$-ideal.
A basic fact of commutative algebra is that describing $I$ (or equivalently $R/I$) is equivalent to describing
\begin{itemize}
    \item its \DEF{associated primes} $\Ass(R/I)$ and
    \item localizations $I_\pp$ for every $\pp\in\Ass(R/I)$.  
\end{itemize}
This short article explains how to use the concept of \DEF{local dual spaces} to describe the latter and pinpoints the data necessary for \DEF{differential (primary) decomposition} of $I$ in terms of local dual spaces.

\smallskip

A thorough study into the early history focused on Macaulay's research %
\cite[\S 30.4]{mora2005solving} sheds light on the involvement of Lasker, Noether, Macaulay, and then later Gr\"obner in the development of the key concepts we are about to discuss. In many articles \emph{bases} of certain local dual spaces are referred to as \emph{(Macaulay) dual bases} or \emph{(Macaulay) inverse systems}. 
Sir F.\,S.\,Macaulay himself used the term \emph{modular forms} in the seminal 1918 book~\cite{macaulay1994algebraic}. 

The language of dual spaces found various applications. One example is that to polynomial systems with isolated singular solutions as in several papers coauthored by Szanto~\cite{hauenstein2017deflation,mantzaflaris2020punctual,mantzaflaris2023certified}.
Another example comes from the area of partial differential equations and was popularized in \cite{sturmfels2002solving,cid2021primary,el2023linear} among applied algebraists: \emph{Ehrenpreis's fundamental principle} (also known as the \emph{fundamental principle of Ehrenpreis and Palamodov}) makes it possible to write down a general solution to a system of PDEs with constant coefficients as long as the irreducible decomposition of the so-called \emph{characteristic variety} is known.

The latter application relates to the notion of \DEF{Noetherian operators} that became an object of several recent studies~\cite{cid2021primary,chen2022noetherian,cid2023primary} with computational techniques implemented in computer algebra system Macaulay2~\cite{chen2023noetherian,wwwM2}.

By the end of reading this article, the reader should understand that Noetherian operators that one may use to describe primary components are yet another incarnation of elements in local dual spaces that correspond to the associated primes.

\section{Local dual space}
The linear differential operators with polynomial coefficients form the \DEF{Weyl algebra} $W 
= R \langle \p_\xx \rangle 
= \KK[x_1,\dots,x_n] \langle \p_1, \ldots, \p_n \rangle$. We shall forget about multiplication (i.e., composition) of differential operators in $W$, but retain \emph{two} $R$-module structures: for $f\in R$ and $D := \sum c_\alpha \p^\alpha \in W$ we have  
\begin{itemize}
    \item the left action $fD := \sum (fc_\alpha) \p^\alpha$, and
    \item the right action: $D\cdot f$ is the differential operator that multiplies the input by $f$ before applying $D$.    
\end{itemize}
The relation between the two actions is given by
$$\p_i \cdot x_i = x_i\p_i + 1 \text{ and } \p_i \cdot x_j = x_j\p_i,\text{ for } i\neq j.$$

\subsection{Definition}

For an $R$-algebra $A$, let $W_A := A \otimes_R W$.
Let $\mm$ denote a maximal ideal in $R$, and $\kappa(\mm) := R/\mm$ the \DEF{residue field} at $\mm$.
We call $W_{\kappa(\mm)}$ the \DEF{local differential space} (at $\mm$).
In what follows $W_{\kappa(\mm)}$ is perceived as a left $\kappa(\mm)$-vector space and a right $R$-module.
\begin{remark}\label{ex:multiplication-derivation}
Note that the \emph{right} multiplication may be perceived as \emph{anti-differentiation} if $\mm$ corresponds to a \DEF{rational point} $(p_1,\dotsc,p_n) \in \KK^n$.
Indeed, multiplying an operator $D$ by $(x_i-p_i)$ on the right has the effect of taking a derivative of $D$ with respect to $\p_i$.
\end{remark}

We define the \DEF{local dual space} of $I$ at $\mm$ as
\begin{align*}
  D_\mm[I] := \{ D \in W_{\kappa(\mm)} \mid D\cdot f \in \mm, \,\forall f \in I \},
\end{align*}
i.e. the set of all operators in $W_{\kappa(\mm)}$ that vanish at $\mm$, when applied to any polynomial in $I$.
For a rational point $\mm = \ideal{x_1-p_1,\dots,x_n-p_n}$, the residue field $\kappa(\mm)$ is isomorphic to the ground field $\KK$, and the local dual space is known as the \emph{(Macaulay) inverse system} of $I$.

For a general (i.e. nonrational) $\mm$, the local dual space $D_\mm[I]$ inherits the structure of its ambient $W_{\kappa(\mm)}$. It is 
\begin{itemize}
    \item a $\kappa(\mm)$-vector space via multiplication on the left; 
    \item a right $R$-module.
\end{itemize}

\begin{remark}\label{ex:primary-finite-dim} 
    Note that if $I$ is $\mm$-primary, then $D_\mm[I]$ is finite-dimensional as a $\kappa(\mm)$-vector space.
    Indeed, $I \supset \mm^d$ for some $d$, which implies that the order of $D\in D_\mm[I]$ is bounded by $d$. 
\end{remark}

\subsection{Local dual space describes localization}\label{sec:LDS-describes-localization}
Consider the localization $R_\mm$ of the polynomial ring $R$ at a maximal ideal $\mm$. We can describe the localization of $I$ by its contraction to $R$ that, in turn, can be seen as an (infinite) intersection of nested $\mm$-primary ideals.

It is a good exercise to show that
\begin{equation}\label{eq:infinite-intersection}
    I_\mm \cap R = \bigcap_{d \ge 1} (I + \mm^d)
\end{equation}
On the other hand, introduce 
\DEF{truncated} local dual spaces $D^{(d-1)}_\mm[I] := D_\mm[I+\mm^d]$, which collect operators in $D_\mm[I]$ of order at most $d-1$.
Then the local dual space is the infinite sum of these nested spaces:
\begin{equation}\label{eq:infinite:sum}
    D_\mm[I] = \sum_{d \ge 1} D^{(d-1)}_\mm[I].
\end{equation}

\begin{lemma}\label{ex:truncation} 
For any $I \subseteq R$ and $d \ge 1$, there is an equality
\[
I + \mm^d = \{ f \in R \mid D \cdot f \in \mm, \, \forall D \in D^{(d-1)}_\mm[I] \}.
\]
\end{lemma}
\begin{proof}
  We invite the reader to show this statement in case of a rational point.
  One may also apply \cite[Theorem 2.6]{marinari1993grobner} to the $\mm$-primary ideal $P=I+\mm^d$ in the rational case.

  That Theorem is followed by discussion of the nonrational case: let $\bar \KK$ be the algebraic closure of field $\KK$; then the extension of $P$, an $\mm$-primary ideal, in $\bar\KK[\xx]$ equals the intersection of primary ideals supported at rational points. This results in the Theorem upgraded to a more general \cite[Proposition 2.7]{marinari1993grobner}, from which the conclusion to our Lemma in bthe general case follows. 
\end{proof}

Since truncations of $D_\mm[I]$ capture $I+\mm^d$ for all $d>0$, in view of \Cref{eq:infinite-intersection}, $D_\mm[I]$ captures the localization of $I$ at $\mm$.

\begin{remark}
  We leave it to the reader to verify that for a pair of ideals $I$ and $J$ the natural properties,
  \begin{equation}
    D_\mm[I+J] = D_\mm[I] \cap D_\mm[J] \quad\text{and}\quad D_\mm[I\cap J] = D_\mm[I] + D_\mm[J],
  \end{equation}
  hold as expected.
\end{remark}

\subsection{Reduction to dimension 0}\label{sec:reduction-to-dim-0}
For a prime $\pp \subseteq R$, one may choose a \DEF{regular system of parameters} $\tt \sqcup \yy$ for $R$ such that $\tt$ are \DEF{free variables}, i.e., a maximal set of elements whose cosets are  algebraically independent in $R/\pp$. 

Let us denote $$\rat{\tt}{\square} := \KK(\tt)\otimes_{\KK[\tt]}\square$$ the result of tensoring with the field of rational 
functions $\KK(\tt)$ (i.e. localizing at the multiplicative set $\KK[\tt] \setminus \{0\}$).
We consider $\rat{\tt}{R} = \KK(\tt)\otimes R \cong \KK(\tt)[\yy]$ and note that $\rat{\tt}{R}\subseteq R_\pp \subseteq \Frac(R)$.

\begin{remark}\label{ex:extension-of-p-is-maximal} 
The ideal $\rat{\tt}{\pp}$ is maximal in $\rat{\tt}{R}$ and we have  $I_\pp = \rat{\tt}{I}_{\rat{\tt}{\pp}}$, i.e., these are equal as subsets of $\Frac(R)$.

\end{remark}

To describe $I_\pp$, therefore, it suffices to describe $D_{\rat{\tt}{\pp}}[\rat{\tt}{I}]$.
By passing from $R$ to $\rat{\tt}{R}$ we may assume that the prime ideal $\rat{\tt}{\pp}$ is maximal.

\section{Primary decomposition}
The underlying goal of any algorithmic approach to primary decomposition of an ideal $I\subseteq R$ is to dissect the underlying affine scheme: find the associated primes $\Ass(R/I)$ and describe the localizations $I_\pp$ for all $\pp\in\Ass(R/I)$. That is --- in geometric terms --- ``decompose'' into components along which the local description of $I$ is well understood. 

\subsection{Classical primary decomposition}
Classically this goal is achieved by finding $\pp$-primary ideals $Q^{(\pp)}$ for each $\pp \in \Ass(R/I)$, such that 
\[
I = \bigcap_{\pp\in\Ass(R/I)} Q^{(\pp)}\,.
\]
Hence the word ``primary'' in ``primary decomposition''. The ideals $Q^{(\pp)}$ are determined uniquely only for minimal primes $\pp$.

\subsection{Differential primary decomposition}
Let $\mm$ be a maximal ideal of $R$.
Define the \DEF{excess dual space}
\begin{equation}\label{eq:finite-dual-quotient}
E_\mm[I] := \frac{D_{\mm}[I]} {D_\mm[I:\mm^\infty]},
\end{equation}
as the quotient $\kappa(\mm)$-vector space characterizing the extent to which $D_\mm[I]$ exceeds $D_\mm[I:\mm^\infty]$. 

\begin{lemma}\label{ex:excess-dual-has-finite-dim}
  The dimension of $E_\mm[I]$ is finite. %
\end{lemma}
\begin{proof}
      Let $\mathcal{A} = \Ass(R/I)$, and $I = \bigcap_{\pp \in \mathcal{A}} Q^{(\pp)}$ be a primary decomposition.
    A primary decomposition of the saturation is given by
    \begin{align*}
    I \colon \mm^\infty = \bigcap_{\pp \in \mathcal{A} \setminus \{\mm\}} Q^{(\pp)}
    \end{align*}
    Since the local dual space at $\mm$ describes the localization, thus ignoring the primary components not contained in $\mm$, we can rewrite the numerator and denominator in the definition of $E_\mm[I]$ below.
    \begin{align*}
      E_\mm[I] = \frac{D_\mm\left[\bigcap_{\pp \in \mathcal{A} \colon \pp \subseteq \mm} Q^{(\pp)} \right]}{D_\mm\left[\bigcap_{\pp \in \mathcal{A} \colon \pp \subsetneq \mm} Q^{(\pp)} \right]} = 
      \frac{D_\mm\left[ Q^{(\mm)} \right] + D_\mm\left[\bigcap_{\pp \in \mathcal{A} \colon \pp \subsetneq \mm} Q^{(\pp)} \right]}{D_\mm\left[\bigcap_{\pp \in \mathcal{A} \colon \pp \subsetneq \mm} Q^{(\pp)} \right]}
    \end{align*}
    concluding that $E_\mm[I] = D_\mm[Q^{(\mm)}] / D_\mm[I \colon \mm^\infty]$.
    Since $D_\mm[Q^{(\mm)}]$ is a finite-dimensional $\kappa(\mm)$-vector space, so is $E_\mm[I]$.
\end{proof}

\begin{theorem}\label{main-theorem}
  Suppose for each $\pp\in\Ass(R/I)$ we have the following data
  \begin{enumerate}
      \item a regular system of parameters $\tt \sqcup \yy$ of $R$ such that $\tt$ are the free variables for $\pp$ and 
      \item a finite set of differential operators $\excessOps_\pp \subseteq D_{\rat{\tt}{\pp}}[\rat{\tt}{I}] \subseteq \kappa(\rat{\tt}{\pp})\langle\p_\yy\rangle$ whose cosets span $E_{\rat{\tt}{\pp}}[\rat{\tt}{I}]$.
  \end{enumerate}
  
  Then the ideal $I$ can be recovered from $\Ass(R/I)$ and this data.  
\end{theorem}
\begin{proof}
Recall that it is sufficient to recover the localization $I_\pp$ for every associated prime to reconstruct the ideal: $$I = \bigcap_{\pp\in\Ass(R/I)} \left(I_\pp \cap R\right).$$

Consider $\Ass(I_\pp\cap R)$, i.e., all associated primes that are contained in~$\pp$.  
Suppose that the localization $I_{\pp'}$ for every $\pp'\in\Ass(I_\pp\cap R)$ such that $\pp'\neq \pp$ has been described.
Then the local dual space  $D_{\rat{\tt}{\pp}}[\rat{\tt}{I}:(\rat{\tt}{\pp})^\infty]$ can be recovered following \Cref{sec:LDS-describes-localization}.

We conclude that $$D_{\rat{\tt}{\pp}}[\rat{\tt}{I}] = D_{\rat{\tt}{\pp}}[\rat{\tt}{I}:(\rat{\tt}{\pp})^\infty] + \kappa(\rat{\tt}{\pp})\excessOps_\pp$$ by the definition of the excess dual space. Thus we can recover $I_\pp$. 
\end{proof}

\Cref{main-theorem} results directly in the following (global) \DEF{membership test}.
\begin{corollary}
 A polynomial $f\in R$ belongs to $I$ if and only if 
\[
D \cdot [f] = 0 \in  \kappa(\rat{\tt}{\pp})
\]
for each $\pp\in\Ass(R/I)$ and each $D\in\excessOps_\pp$ where  $\excessOps_\pp$ is described in~\Cref{main-theorem}.
\end{corollary}

\begin{remark}\label{ex:diff-primary-decomposition} 
For each associated prime $\pp$ of $I\subset R=\KK[\xx]$, choose a partition $\tt \sqcup \yy =\xx$, that is, choose free variables from the original variables. Then
\begin{enumerate}[(i)]
\item one may pick finitely many differential operators ${\mathfrak A}_\pp$ in $R \langle \p_\yy \rangle \subseteq W$
whose images $A_\pp$ in $W_{\kappa(\rat{\tt}{\pp})}$ form a basis of $E_{\rat{\tt}{\pp}}[\rat{\tt}{I}]$, and
\item we have $ 
    I = \{f \in R \mid D \dot f \in \pp \text{ for all } \pp\in \Ass(R/I),\, D\in {\mathfrak A}_\pp \}.
    $
\end{enumerate}
\end{remark}

The operators in ${\mathfrak A}_\pp$ are sometimes called \emph{Noetherian operators}; the reader may follow~\cite{cid2023primary} that unpacks the above remark --- albeit in a different language --- and connects the dimension of the excess dual space to the notion of arithmetic multiplicity.

\section{Example}
The following result provides a unique choice of a primary component for each associated prime, including the embedded primes.

For a primary ideal $Q$ define $\nil(Q) = \min\{m \mid (\sqrt Q)^m \subseteq Q\}$.
\begin{theorem}[Ortiz~\cite{ortiz1959certaine}]
\label{thm:Ortiz}
  Every ideal $I \subset R$ admits a unique irredundant primary decomposition $I = Q_1 \,\cap\,\dots\,\cap\, Q_r$
  with the following property: if $I=Q_1' \,\cap\,\dots\,\cap\, Q_r'$ is another irredundant primary decomposition,
  with $\sqrt{Q_i} = \sqrt{Q_i'}$ for $i=1,\dots,r$,  then for all $i$,
  \begin{enumerate}
    \item $\nil(Q_i) \leq \nil(Q_i')$, and
    \item if $\nil(Q_i) = \nil(Q_i')$, then $Q_i \subseteq Q_i'$.
  \end{enumerate}
\end{theorem}

The excess dual space gives a convenient way to obtain the unique irredundant primary decomposition guaranteed by \Cref{thm:Ortiz}. We demonstrate this on an example of 
$$
I = \ideal{x_1-x_2^3} \cap \ideal{x_2-x_1^3} \cap \ideal{x_1^3,x_2^3,x_1^2 x_2 - x_1 x_2^2}.  
$$
Taking $Q_1 = \ideal{x_1-x_2^3}$ and $Q_2 = \ideal{x_2-x_1^3}$ for the minimal primes, we would like to construct $Q_3$, an $\mm$-primary ideal for $\mm = \ideal{x_1,x_2}$ with the properties required in \Cref{thm:Ortiz}.

We can compute the excess dual space
$E_\mm[I] = D_\mm[I] / D_\mm[Q_1\cap Q_2]$
by iteratively obtaining the truncation
$$E^{(d)}_\mm[I] = D^{(d)}_\mm[I] / D^{(d)}_\mm[Q_1\cap Q_2]$$ until $E^{(d+1)}_\mm[I] = E^{(d)}_\mm[I]$. \Cref{ex:excess-dual-has-finite-dim} guarantees termination.

Truncated dual spaces can be obtained for instance with \verb|truncatedDual| method in \verb!NoetherianOperators! package \cite{chen2023noetherian} of Macaulay2 \cite{grayson2002macaulay2}. 
This computation produces two representatives, 
$\p_1\p_2$ and $\p_1^2 \p_2 + \p_1 \p_2^2$, 
for the cosets that form a basis of $E_\mm[I]$.
Now, knowing that $\nil Q_3 = 4$ and $$
D_\mm^{(3)}[I] = \Span\{1,\p_1,\p_2,\p_1^2,\p_1\p_2,\p_2^2,\p_1^3,\p_2^2,\p_1^2 \p_2 + \p_1 \p_2^2\}\,,
$$
we conclude $Q_3 = \ideal{x_1^2 x_2 - x_1 x_2^2} + \mm^4$. 

\section{Conclusion and future}
The concept of \DEF{local dual space} naturally generalizes the more than 100-year-old (known to Macaulay) concept of \emph{inverse system} defined at a rational point. It appears to be the most convenient for working with complex affine schemes in the framework of \DEF{numerical algebraic geometry}. 
The authors worked on the material in the current article a few years ago with the original intent of advancing the program of \DEF{numerical primary decomposition} that seeks to describe the underlying affine scheme in terms of the so-called \DEF{witness sets} for associate primes. While the job can be accomplished in absense of embedded components (see ~\cite{chen2022noetherian,chen2023noetherian}), the question of efficient computation for embedded components remains wide open. It remains challenging not only to obtain witness points on the varieties corresponding to embedded primes (see \cite{leykin2008numerical} and \cite{krone2017numerical} for existing algorithms) but also to determine the corresponding \DEF{excess dual spaces} (defined in this article).
Note that the purely symbolic algorithm described in~\cite{cid2023primary} is of a different nature. There, de facto, the classical primary decomposition is assumed to be known in order to derive a differential primary decomposition. However, in the framework of numerical algebraic geometry, one cannot assume that generators of a primary ideal in the classical decomposition are known: the original generators of the ideal defining the scheme are never rewritten or manipulated in a manner algorithms based on Gr\"obner bases do.   

We hope that in the future these challenges would be addressed making a hybrid numerical-symbolic approach to differential primary decomposition the most practical method for experimentation with ideals arising both in pure and applied algebraic geometry. 

\bibliography{references}
\bibliographystyle{plain}

\end{document}